\documentclass[11pt,reqno]{amsart}

\usepackage[a4paper,margin=28mm]{geometry}
\usepackage{color}
\usepackage[T1]{fontenc}
\usepackage{lmodern}
\usepackage[nopatch=footnote]{microtype}
\usepackage{amsmath,amssymb,mathtools,mathrsfs}
\usepackage[hidelinks]{hyperref}

\numberwithin{equation}{section}

\newtheorem{theorem}{Theorem}[section]
\newtheorem{proposition}[theorem]{Proposition}

\newtheorem{corollary}[theorem]{Corollary}

\theoremstyle{definition}

\theoremstyle{remark}
\newtheorem{remark}[theorem]{Remark}

\newcommand{\Teich}{\mathcal T_g}
\newcommand{\CH}{\mathbb C\mathbb H^2}
\newcommand{\PU}{\operatorname{PU}(2,\,1)}
\newcommand{\X}{\mathcal X_0^{\mathrm{good}}}
\newcommand{\CC}{\mathcal C_0^{\mathrm{cc}}}
\newcommand{\OO}{\mathcal O}
\newcommand{\RR}{\mathbb R}
\newcommand{\CCplx}{\mathbb C}
\newcommand{\Sym}{\operatorname{Sym}}
\newcommand{\Ext}{\operatorname{Ext}}
\newcommand{\Tot}{\operatorname{Tot}}
\newcommand{\Crit}{\operatorname{Crit}}
\newcommand{\Area}{\operatorname{Area}}
\newcommand{\ram}{\operatorname{ram}}
\newcommand{\dd}{\,\mathrm d}
\newcommand{\Mimm}{\mathscr M_0^{\mathrm{imm}}}
\newcommand{\Hgood}{\mathscr H_0^{\mathrm{good}}}

\title[Ramification and finiteness of minimal surfaces]
{Ramification and Finiteness of Equivariant Minimal Surfaces}
\author[I. Biswas]{Indranil Biswas}
\address{Department of Mathematics, Shiv Nadar University,
Dadri 201314, Uttar Pradesh, India}
\email{indranil.biswas@snu.edu.in}
\author[S. Guin]{Satyajit Guin}
\address{Department of Mathematics and Statistics,
Indian Institute of Technology Kanpur, Kanpur 208016, India}
\email{sguin@iitk.ac.in}
\author[P. Kumar]{Pradip Kumar}
\address{Department of Mathematics, Shiv Nadar University,
Dadri 201314, Uttar Pradesh, India}
\email{pradip.kumar@snu.edu.in}
\date{}


\begin{document}

\begin{abstract}
Let $S$ be a closed oriented surface of genus $g$, with $g\ \geq \ 2$, and let
$\X$ be the good locus of $\PU$ representations with Toledo
invariant zero. For $1\ \leq \ b\ \leq\ 2g-3$, we show that representations
admitting an equivariant minimal map of total ramification at least
$b$, for some marked complex structure, form a countable union of
compact null sets of Hausdorff dimension at most $16g-16-2b$.
For larger $b$, this locus is empty. Consequently, on a dense
$G_\delta$ subset of full measure in $\X$, every nonconstant
equivariant minimal map is an immersion, for every marked complex
structure. On the convex cocompact locus, a dense $G_\delta$ subset
of full measure consists of representations admitting a finite
nonempty set of marked equivariant minimal immersions.
\end{abstract}
\maketitle

\section{Introduction}

\label{sec:introduction}
Let $S$ be a closed connected oriented smooth surface of genus
$g\,\geq\,2$, and put $\Gamma\,:=\,\pi_1(S,\,s_0)$ for a fixed
base point $s_0$. We study minimal maps from the universal cover
$\widetilde S$ into the complex hyperbolic plane $\CH$ which are
equivariant for a representation
$\rho\,:\,\Gamma\,\longrightarrow\,\PU$. The complex structure on
$S$ is allowed to vary. Our aim is to investigate the representations
for which such a map can have branch points, and to prove finiteness
of the set of marked minimal immersions for generic convex cocompact
holonomy with Toledo invariant zero.

\smallskip

The relationship between surface group representations and holomorphic
Higgs bundles is given by nonabelian Hodge theory. The work of Hitchin,
Donaldson, Corlette, and Simpson provides the harmonic map and Higgs
bundle descriptions used here \cite{Hit87,Don87,Cor88,Sim92}. For the
complex hyperbolic plane, Loftin and McIntosh described equivariant
minimal surfaces which are neither holomorphic nor antiholomorphic by
two effective divisors and an extension of holomorphic line bundles
\cite[Theorem~2.3]{LM19}. The divisors give
the zeros of the two components of the Higgs field, and their common
part records branching. McIntosh subsequently proved smoothness of
the moduli space of minimal immersions which are neither holomorphic
nor antiholomorphic, and determined its connected components
\cite[Theorems~1.1 and~1.2]{McI23}. These descriptions provide the
starting point for our argument.

\medskip

We specify the setting. Let $\Teich$ be the space of complex
structures on $S$ compatible with its orientation, modulo
orientation-preserving diffeomorphisms isotopic to the identity.
This is the Teichm\"uller space of $S$, and it has complex dimension
$3g-3$. For $J\,\in\,\Teich$, write $S_J$
for the corresponding Riemann surface and $K_J$ for its canonical
line bundle. We equip $\CH$ with the homogeneous K\"ahler metric
$g_{\CH}$ of holomorphic sectional curvature $-4$ and denote its
K\"ahler form by $\omega$. The real sectional curvatures then lie
in $[-4,\,-1]$.

A representation $\rho\,:\,\Gamma\,\longrightarrow\,\PU$ is called \emph{good} if it is
reductive and irreducible and the centralizer of $\rho(\Gamma)$ in
$\PU$ is trivial. In our normalization of holomorphic sectional curvature of $\CH$, the Toledo invariant is given by 
\begin{equation}
\label{eq:toledo}
  \tau(\rho)\,\ =\,\ \frac{2}{\pi}\int_S f^*\omega,
\end{equation}
where $f\,:\,\widetilde S\,\longrightarrow\,\CH$ is any smooth
$\rho$-equivariant map. The form $f^*\omega$ descends to $S$, and
the integral is independent of the choice of $f$
\cite[Section~2]{LM19}. Consider the reductive character variety
$$
  \mathcal R(\Gamma,\,\PU)
  \,:=\,\operatorname{Hom}_{\mathrm{red}}(\Gamma,\,\PU)/\PU,
$$
where $\operatorname{Hom}_{\mathrm{red}}$ denotes reductive
homomorphisms and the quotient is by conjugation. Denote the locus
\begin{equation}
\label{eq:good-locus}
  \X\,:=\,
  \bigl\{[\rho]\,\in\,\mathcal R(\Gamma,\,\PU)
  \,\big\vert\,\tau(\rho)\,=\,0,\ \rho\text{ is good}\bigr\}.
\end{equation}
This is a smooth real manifold of dimension
\begin{equation}
\label{eq:ambient-dim}
  N\,:=\,\dim_{\RR}\X\,=\,16g-16,
\end{equation}
by the standard deformation theory of surface-group
representations \cite[Sections~1.2--1.5]{Gol84}.
The dimension calculation is briefly recalled in
Section~\ref{subsec:character-proof}.
\medskip 

For each $(J,\,[\rho])\,\in\,\Teich\times\X$, there is a unique
$\rho$-equivariant harmonic map $f_{J,\rho}$. It is nonconstant by
irreducibility. We call it an \emph{equivariant minimal map} when it
is weakly conformal, allowing isolated branch points. Its branch
divisor and total ramification are
$$
  B(f)\,=\,\sum_{p\in S}\operatorname{ord}_p(df)\,p,
  \qquad
  \ram(f)\,=\,\deg B(f),
$$
where $\operatorname{ord}_p(df)$ is the branching order. Their
description in terms of Higgs fields is given in
Proposition~\ref{prop:extension}. A map with $B(f)\,=\,0$ is a
minimal immersion.

We retain the marking throughout. More precisely, two triples $(J,\,\rho,\,f)$ and $(J',\,\rho',\,f')$ are identified if there exist an orientation-preserving diffeomorphism $\phi\,:\,S\,\to\,S$ isotopic to the identity, a lift $\widetilde{\phi}:\widetilde S\to\widetilde S$, and an element $A\,\in\,\PU$ such that
$$
\begin{aligned}
J' &= \phi_*J,\\
\rho'(\gamma) &= A\,\rho\bigl(\widetilde{\phi}^{-1}\gamma\widetilde{\phi}\bigr)A^{-1},
\qquad \gamma\in\Gamma,\\
f' &= A\circ f\circ\widetilde{\phi}^{-1}.
\end{aligned}
$$
Here $\Gamma$ is identified with the group of deck transformations of $\widetilde S$, and $\phi_*J$ denotes the complex structure transported by $\phi$. By uniqueness of the harmonic map,
the space $\Mimm$ of these marked minimal immersions is a subspace
of $\Teich\times\X$, with holonomy projection
\begin{equation}
\label{eq:immersion-projection}
  \pi_{\mathrm{imm}}\,:\,\Mimm\,\longrightarrow\,\X,
  \qquad (J,\,[\rho])\,\longmapsto\,[\rho].
\end{equation}
We use the subspace topology, as in \cite[Section~2]{McI23}.

For an integer $b\,\geq\,1$, let
\begin{equation}
\label{eq:exceptional}
  \mathcal B_b\,:=\,
  \left\{[\rho]\,\in\,\X\,\middle\vert\,
  \begin{array}{l}
    \text{some marked equivariant minimal map }f\\
    \text{with holonomy }[\rho]\text{ has }\ram(f)\,\geq\,b
  \end{array}\right\}.
\end{equation}
We compute
Hausdorff dimension with respect to any smooth Riemannian metric on
$\X$. A subset is called
\emph{null} if it has zero volume for every smooth metric, and
\emph{meagre} if it
is a countable union of nowhere dense subsets. As usual, $F_\sigma$
and $G_\delta$ mean a countable union of closed subsets and a countable
intersection of open subsets, respectively.

\begin{theorem}
\label{thm:ramification}
For every integer $1\,\leq\,b\,\leq\,2g-3$,
\begin{equation}
\label{eq:main-bound}
  \dim_{\mathrm H}\mathcal B_b\,\leq\,16g-16-2b\,=\,N-2b.
\end{equation}
The subset $\mathcal B_b$ is a countable union of compact null sets.
In particular, it is $F_\sigma$, null, and meagre in $\X$.
If $b\,>\,2g-3$, then $\mathcal B_b\,=\,\varnothing$.
\end{theorem}

\begin{corollary}
\label{cor:generic-immersion}
The complement $\X\setminus\mathcal B_1$ is a dense $G_\delta$
subset of full measure in $\X$. For each representation in this
complement, every nonconstant equivariant weakly conformal harmonic
map is an immersion, for every $J\,\in\,\Teich$.
\end{corollary}

\medskip 
We now restrict to convex cocompact holonomy. A representation is
called \emph{convex cocompact} if it
is an isomorphism onto a discrete subgroup acting cocompactly on a
nonempty invariant closed convex subset of $\CH$. Put
\begin{equation}\label{eq:cc-locus}
  \CC\ :=\ \bigl\{[\rho]\,\in\,\X\,\,\,\big\vert\,\,\,
  \rho\text{ is convex cocompact}\bigr\}.
\end{equation}
This locus is open by the rank-one characterization and openness of
Anosov representations \cite[Theorems~5.13 and~5.15]{GuiW12}.

The energy of the harmonic map defines a function on $\Teich$,
\begin{equation}
\label{eq:energy}
  E_\rho(J)\ =\ \frac12\int_S
  |df_{J,\rho}|^2_{\sigma_J,g_{\CH}}\dd A_{\sigma_J},
\end{equation}
where $\sigma_J$ is any metric in the conformal class $J$.
Conformal invariance in dimension two makes this independent of
$\sigma_J$. The first-variation formula identifies its critical
points with equivariant minimal maps \cite{DW07}. We prove:

\begin{theorem}
\label{thm:finiteness}
There is a subset $\mathcal G\,\subset\,\CC\setminus\mathcal B_1$
which is dense and $G_\delta$ in $\CC$ and has full measure there,
such that, for every $[\rho]\,\in\,\mathcal G$,
\begin{equation}
\label{eq:finite-fiber}
  1\,\leq\,\#\Crit(E_\rho)
  \,=\,\#\pi_{\mathrm{imm}}^{-1}([\rho])\,<\,\infty.
\end{equation}
In particular, the set of marked $\rho$-equivariant minimal
immersions is finite and nonempty.
\end{theorem}

Section~\ref{sec:background} gives the analytic input and the
extension description, including the branching multiplicities.
Section~\ref{sec:incidence} constructs the incidence spaces and
proves Theorem~\ref{thm:ramification} and
Corollary~\ref{cor:generic-immersion}. Section~\ref{sec:energy}
proves Theorem~\ref{thm:finiteness} using the energy and compactness
arguments.

\section{Nonabelian Hodge theory}\label{sec:background}

\subsection{The character locus}
\label{subsec:character-proof}

As before, $S$ is a closed connected oriented surface of genus $g$, with
$g\, \geq\, 2$, and $\rho\, :\, \Gamma\, :=\, \pi_1(S,\, s_0)\, \longrightarrow\,
\PU$ is a homomorphism with $[\rho]\,\in\,\X$. Denoting the Lie algebra of $\PU$ by $\mathfrak{pu}(2,1)$,
let $\operatorname{Ad}\rho$ be the flat real vector
bundle over $S$ associated with the adjoint representation
$\operatorname{Ad}\circ\rho\,:\,\Gamma\,\longrightarrow\,
\operatorname{GL}(\mathfrak{pu}(2,1))$, namely
$$
\operatorname{Ad}\rho
\,\ :=\,\ \widetilde S\times^{\operatorname{Ad}\circ\rho}
\mathfrak{pu}(2,1).
$$

The tangent space to the good character variety at
$[\rho]$ is naturally identified with
$$
T_{[\rho]}\X\ \cong\ H^1(S;\,\operatorname{Ad}\rho)
$$
(see \cite[Sections~1.2--1.5]{Gol84}). Since $\rho$ is good, the centralizer of $\rho(\Gamma)$ in $\PU$
is trivial. Consequently,
$H^0(S;\, \operatorname{Ad}\rho)\,=\,0$. The nondegenerate invariant Killing form on
$\mathfrak{pu}(2,1)$ implies that  $\operatorname{Ad}\rho\,=\,(\operatorname{Ad}\rho)^*$.
Hence Poincar\'e duality gives that $H^2(S;\operatorname{Ad}\rho)
\,\simeq\,H^0(S;\, \operatorname{Ad}\rho)^*\,=\,0$. Now, by the formula for Euler characteristic,
$$
\begin{aligned}
\dim_{\RR}H^1(S;\operatorname{Ad}\rho)
\ &=\ -\chi(S)\dim_{\RR}\mathfrak{pu}(2,1)\\
&=\ (2g-2)\cdot 8\\
&=\ 16g-16.
\end{aligned}
$$
Since the Toledo invariant is a locally constant function,
it follows that $\X$ is a smooth real manifold of
dimension $16g-16$.

\subsection{Harmonic maps and universal Higgs bundles}

Let $\Hgood$ be the universal moduli space over $\Teich$ of
stable projective $\PU$-Higgs bundles whose associated
representations are good and have Toledo invariant zero.
The following proposition gives the regularity needed below.

\begin{proposition}
\label{prop:analytic-inputs}
The nonabelian Hodge correspondence gives a smooth identification
\begin{equation}
\label{eq:universal-hodge}
\mathcal N\ :\ \Hgood\ \xrightarrow{\,\ \simeq\,\ }
\Teich\times\X.
\end{equation}
The space $\Mimm$ is a second-countable smooth complex manifold
of complex dimension $8(g-1)$.
\end{proposition}

\begin{proof}
On the good locus, the universal nonabelian Hodge correspondence
is smooth, with smooth inverse; see \cite[Section~2]{McI23}. 
Since the Toledo invariant is locally constant, restricting this
correspondence to the Toledo-zero locus gives
\eqref{eq:universal-hodge}.

By \eqref{eq:toledo}, a nonconstant holomorphic equivariant map
has positive Toledo invariant, while a nonconstant
antiholomorphic equivariant map has negative Toledo invariant.
Thus neither type occurs in $\Mimm$. We may therefore apply
\cite[Theorem~1.1]{McI23}, which gives a smooth complex manifold
of dimension
$$
(g-1)\dim_{\RR}\PU\ =\ 8(g-1).
$$
The restriction to good holonomy is open. The topology of this
manifold agrees with the subspace topology in
$\Teich\times\X$; see \cite[Section~2 and Lemma~5.1]{McI23}.
Since $\Teich\times\X$ is second countable, so is $\Mimm$.
\end{proof}

\begin{remark}\label{eq:critical-conformal}
We also recall the variational interpretation that will be used in
Section~\ref{sec:energy}. Let $q_{J,\rho}$ be the Hopf differential of
$f_{J,\rho}$. The following three statements are equivalent:
\begin{enumerate}
\item $J\ \in\ \Crit(E_\rho)$,

\item $q_{J,\rho}\ =\ 0$,

  \item $f_{J,\rho}$ is weakly conformal.
\end{enumerate}
See \cite{DW07}. Since the map $f_{J,\rho}$ is
nonconstant, the last condition precisely says that $f_{J,\rho}$ is an
equivariant minimal map in our convention.
\end{remark}

\subsection{Extension data and the branch divisor}
\label{sec:extensions}

As in \cite[Section~2]{LM19}, we can normalize a projective
$\PU$-Higgs bundle by tensoring it with the inverse of its rank-one
summand. In this normalization, it takes the form
\begin{equation}
\label{eq:higgs-form}
  E=V\oplus\OO,
  \qquad
  \Phi=
  \begin{pmatrix}
    0&\beta\\
    \gamma&0
  \end{pmatrix},
  \qquad
  \beta:\OO\longrightarrow V\otimes K_J,
  \quad
  \gamma:V\longrightarrow K_J,
\end{equation}
where $V$ is a holomorphic vector bundle of rank two. In this set-up,
\begin{equation}
\label{eq:higgs-toledo}
  \tau(\rho)\ =\ -\frac23\deg V.
\end{equation}
The Hopf differential vanishes exactly when
\begin{equation}
\label{eq:conformality}
0\ =\ (\gamma\otimes\operatorname{id}_{K_J})\circ\beta \ \in\ H^0(S_J,\, K_J^2);
\end{equation}
equivalently, $\operatorname{tr}(\Phi^2)\,=\,0$
\cite[Section~2.3]{LM19}. For notational convenience, write $\gamma\beta$ for this composition
$(\gamma\otimes\operatorname{id}_{K_J})\circ\beta$.

The next proposition gives its Toledo-zero form.
We include the degree and multiplicity arguments needed for the
ramification estimate.

\begin{proposition}[{\cite[Theorem~2.3]{LM19}}]\label{prop:extension}
Every marked equivariant minimal map $(J,\, f,\, \rho)$ with holonomy in $\X$ determines
effective divisors $D_1,\,D_2$ of equal degree $d$, where
\begin{equation}\label{eq:d-range}
  0\ \leq\ d\ \leq\ 2g-3,
\end{equation}
and an extension
\begin{equation}
\label{eq:extension}
  0\  \longrightarrow\  K_J^{-1}(D_1)\ \longrightarrow\ V\ \longrightarrow
\ K_J(-D_2)\ \longrightarrow\ 0.
\end{equation}
The extension class is in
\begin{equation}
\label{eq:extension-class}
  \Ext^1\bigl(K_J(-D_2),\,K_J^{-1}(D_1)\bigr)\ \simeq\
  H^1\bigl(S_J,\,K_J^{-2}(D_1+D_2)\bigr).
\end{equation}
Conversely, take any pair of effective divisors of common degree $d$
satisfying \eqref{eq:d-range}, and take any class in \eqref{eq:extension-class}. These
together produce a stable projective $\PU$--Higgs bundle with $\gamma\beta\,=\,0$ and
$\tau\,=\,0$. On the locus with good holonomy, this gives a marked equivariant
minimal map.

In both directions, 
\begin{equation}
\label{eq:branch-divisor}
  B(f)\ =\ \gcd(D_1,\,D_2), \qquad
  \ram(f)\ =\ \deg\gcd(D_1,\,D_2),
\end{equation}
where at each point the multiplicity in the greatest common divisor is
the minimum of the two multiplicities.
\end{proposition}

\section{Incidence spaces and their holonomy maps}
\label{sec:incidence}
\subsection{The relative extension bundle}
\label{subsec:relative-extension-bundle}
Let $\pi\,:\,\mathscr S\,\longrightarrow\,\Teich$ be the universal
marked curve. For $j\,\geq\,1$, denote its relative symmetric
product over $\Teich$ by $\Sym^j(\mathscr S/\Teich)$. Its fiber over any $J$ is
$\Sym^j(S_J)$, a smooth complex manifold of dimension $j$
\cite{Mac62}.  The relative symmetric product is a smooth
holomorphic family of total complex dimension $3g-3+j$. Set
$\Sym^0(\mathscr S/\Teich)\,:=\,\Teich$.
For $d$ in \eqref{eq:d-range}, set
\begin{equation}
\label{eq:divisor-base}
  \mathscr P_d\,=\,
  \Sym^d(\mathscr S/\Teich)
  \times_{\Teich}
  \Sym^d(\mathscr S/\Teich).
\end{equation}
Its points are $(J,\,D_1,\,D_2)$, where $D_1$ and $D_2$ are
effective divisors of degree $d$ on the Riemann surface $S_J$.

\begin{proposition}[{\cite[Section~6.1 and Theorem~6.1]{LM19}}]
\label{prop:extension-bundle}
The extension spaces in \eqref{eq:extension-class}
assemble into a holomorphic vector bundle
\[
\mathscr E_d\ \longrightarrow\ \mathscr P_d
\]
of rank \(r_d=5g-5-2d\).
\end{proposition}

\begin{proof}
Let
$$
q_d\ :\ \mathscr P_d\ \longrightarrow\ \Teich,
\qquad (J,\,D_1,\, D_2)\ \longmapsto\ J
$$
be the natural projection. Consider the universal family of Riemann surfaces $\pi\,:\, \mathscr S
\, \longrightarrow\, \Teich$. Let
\begin{equation}\label{e5}
\mathscr S_d \ :=\  q^*_d \mathscr S\ \stackrel{p_d}{\longrightarrow}\ \mathscr P_d
\end{equation}
be the pullback of this family of Riemann surfaces. Let ${\mathcal D}_1$ (respectively,
${\mathcal D}_2$) be the tautological relative divisor on $\mathscr S_d$ --- for
the projection $p_d$ --- corresponding to the first
(respectively, second) factor of the product $\Sym^d(\mathscr S/\Teich)
\times_{\Teich} \Sym^d(\mathscr S/\Teich)$. Therefore, for any point $z\, :=\, (J,\, D_1,\, D_2)
\, \in\, \mathscr P_d$,
\begin{equation}\label{e7}
\left.\mathcal D_1\right|_{(p_d)^{-1}(z)}\ =\ D_1,\ \ \
\left.\mathcal D_2\right|_{(p_d)^{-1}(z)}\ =\ D_2;
\end{equation}
note that $(p_d)^{-1}(z)\,=\, S_J$. Let
\begin{equation}\label{e8}
K_{p_d}\ :=\ \Omega^1_{\mathscr S_d/\mathscr P_d}\ \longrightarrow\ \mathscr S_d
\end{equation}
be the relative canonical line bundle of this family in \eqref{e5}.

Define the holomorphic line bundle
\begin{equation}
\label{eq:extension-line-bundle}
\mathcal A_d
\ :=\ K_{p_d}^{-2}\otimes
\OO_{\mathscr S_d}(\mathcal D_1+\mathcal D_2)\ \longrightarrow\ \mathscr S_d.
\end{equation}
For any point $z\, :=\, (J,\, D_1,\, D_2)\, \in\, \mathscr P_d$, the restriction
of $\mathcal A_d$ to the fiber $(p_d)^{-1}(z)\,=\, S_J$ is
$K_J^{-2}\otimes {\mathcal O}_{S_J}(D_1+D_2)$. From \eqref{eq:d-range} it follows
immediately that
\begin{equation}\label{e4}
\deg\bigl(K_J^{-2}\otimes {\mathcal O}_{S_J}(D_1+D_2)\bigr)
\ =\ -4g+4+2d\ \leq\ -2.
\end{equation}

Thus the restriction has no nonzero holomorphic section.
Riemann--Roch gives
$$
h^1\bigl(S_J,\,K_J^{-2}(D_1+D_2)\bigr)\ =\ 5g-5-2d,
$$
independently of $z$. Cohomology and base change
\cite{GR84} therefore imply that $(p_d)_*\mathcal A_d\,=\,0$ and
\begin{equation}\label{eq:extension-direct-image}
\mathscr E_d\ :=\ R^1(p_d)_*\mathcal A_d\ \longrightarrow\ \mathscr P_d
\end{equation}
is a holomorphic vector bundle of rank $5g-5-2d$, whose
fiber at $z$ is the extension space in
\eqref{eq:extension-class}.
\end{proof}

\subsection{A common divisor of prescribed degree}

To impose the condition that the total ramification is at least $b$, we choose a common effective
subdivisor $C$ of degree $b$ and write $D_i=C+A_i$ for $i\,=\, 1,\, 2$. The 
divisors $A_i$ have degree $d-b$. Accordingly, for
$1\,\leq\,b\,\leq\,d$, define
\begin{equation}
\label{eq:incidence-base}
  \mathscr I_{d,b}\ =\
  \Sym^b(\mathscr S/\Teich)
  \times_{\Teich}
  \Sym^{d-b}(\mathscr S/\Teich)
  \times_{\Teich}
  \Sym^{d-b}(\mathscr S/\Teich)
\end{equation}
and the holomorphic addition map
\begin{equation}\label{e6}
  a_{d,b}\ :\ \mathscr I_{d,b}\ \longrightarrow\ \mathscr P_d,
  \ \ \  (J,\,C,\,A_1,\,A_2)\ \longmapsto\ (J,\,C+A_1,\,C+A_2).
\end{equation}
Consider the vector bundle $\mathscr E_d$ in \eqref{eq:extension-direct-image}.
Let
\begin{equation}
\label{eq:incidence-total}
  \mathscr F_{d,b}=\Tot(a_{d,b}^*\mathscr E_d)
\end{equation}
be the total space of the pullback of it by the map $a_{d,b}$
in \eqref{e6}. An element is a $5$-tuple $(J,C,A_1,A_2,\xi)$, where $\deg C=b$,
$\deg A_i=d-b$, and $\xi$ is the extension class for $D_i=C+A_i$.

\begin{proposition}
\label{prop:incidence}
The space $\mathscr F_{d,b}$ is a smooth second-countable complex
manifold, and
\begin{equation}
\label{eq:incidence-dim}
  \dim_{\CCplx}\mathscr F_{d,b}=8g-8-b.
\end{equation}
\end{proposition}

\begin{proof}
The relative symmetric products are smooth holomorphic families over
$\Teich$. Their fiber products over $\Teich$ are therefore smooth.
The bundle in \eqref{eq:incidence-total} is
holomorphic. Adding the dimensions gives
\begin{align*}
  \dim_{\CCplx}\mathscr F_{d,b}
  &=
  (3g-3)+b+(d-b)+(d-b)+(5g-5-2d)\\
  &=
  8g-8-b.
\end{align*}
Second countability follows from the fact that both base and fiber
are second countable.
\end{proof}

The parameter $C$ gives a common \emph{subdivisor} of degree $b$.
It does not have to be the whole greatest common divisor. We do not
require $A_1,A_2$ to be disjoint. Thus this construction detects total
ramification \emph{at least} $b$. Here we count multiplicities, not just
the distinct branch points. For instance, if $D_1$ and $D_2$ have
multiplicities three and two at a point $p$, then $p$ contributes
two to the total ramification. The incidence construction can give
either $p$ or $2p$ as part of the chosen common subdivisor.

\subsection{Holonomy maps and the ramification locus}
\label{subsec:tautological-family}

Recall the family $p_d\,:\,\mathscr S_d\,\longrightarrow\mathscr P_d$ in
\eqref{e5}, its universal divisors
(see \eqref{e7}), and the line bundle $\mathcal A_d$
in \eqref{eq:extension-line-bundle}.
Let
\begin{equation}\label{e10}
T_d\ :=\ \Tot(\mathscr E_d)\ \stackrel{q}{\longrightarrow}\ \mathscr P_d
\end{equation}
be the projection from the total space of the vector
bundle $\mathscr E_d$ in \eqref{eq:extension-direct-image}. The base change of $p_d$ along $q$ is
the fiber product
\begin{equation}\label{e9}
p\ :\ \mathscr S_{T_d}\ :=\ \mathscr S_d\times_{\mathscr P_d}T_d\ \longrightarrow\ T_d.
\end{equation}
Let $\widetilde q\,:\,\mathscr S_{T_d}\,\longrightarrow\, \mathscr S_d$ be the
natural projection to the first factor of the fiber product. On $\mathscr S_{T_d}$,
consider the two holomorphic line bundles
$$
\mathcal L\ :=\ \widetilde q^*\bigl(K_{p_d}^{-1}\otimes\OO_{\mathscr S_d}(\mathcal D_1)\bigr),\ \ \
\mathcal Q\ :=\ \widetilde q^*\bigl(K_{p_d}\otimes{\mathcal O}_{\mathscr S_d}  (-\mathcal D_2)\bigr)
$$
(see \eqref{e8} and \eqref{e7}). The pulled back holomorphic line bundle $\mathcal A\,:=\,
\widetilde q^*\mathcal A_d$ is evidently identified with $\mathcal L\otimes\mathcal Q^{-1}$. We have 
$$
  p_*\mathcal A\ =\ 0, \qquad R^1p_*\mathcal A\ \simeq\ q^*\mathscr E_d,
$$
where $p$ and $q$ are constructed in \eqref{e9} and \eqref{e10} respectively.

The pulled back holomorphic vector bundle $q^*\mathscr E_d$ has a tautological holomorphic section:
at a point $t\,=\,(J,\,D_1,\,D_2,\,\xi)\, \in\, T_d$,
this section takes the value $\xi$. Since $p_*\mathcal A\,=\,0$,
the degree-one edge map in the Leray spectral sequence is an
isomorphism. Therefore
$$
\begin{aligned}
\Ext^1_{\mathscr S_{T_d}}(\mathcal Q,\,\mathcal L)
&\ \simeq\ H^1(\mathscr S_{T_d},\,\mathcal A)\\
&\ \simeq\ H^0(T_d,\,R^1p_*\mathcal A)
\ \simeq\ H^0(T_d,\,q^*\mathscr E_d).
\end{aligned}
$$
Thus the tautological section determines a global extension
class, and hence a short exact sequence of holomorphic vector
bundles on $\mathscr S_{T_d}$:
\begin{equation}\label{eq:tautological-extension}
0\ \longrightarrow\ \mathcal L\ \longrightarrow\ \mathcal V\ \longrightarrow\
\mathcal Q\ \longrightarrow\ 0.
\end{equation}
For any point $t\,=\,(J,\,D_1,\,D_2,\,\xi)\, \in\, T_d$, the restriction of the short exact
sequence in \eqref{eq:tautological-extension} to $p^{-1}(t)\,=\, S_J$ is the short exact
sequence of vector bundles on $S_J$ given by the class $\xi$.

Let $K_p\,:=\,\Omega^1_{\mathscr S_{T_d}/T_d}$ be the relative
canonical line bundle. It is naturally identified with
$\widetilde q^*K_{p_d}$. The pulled back tautological divisors
give the canonical maps
$$
\begin{aligned}
\boldsymbol\beta\ :\ &
\OO_{\mathscr S_{T_d}}\ \longrightarrow\
\mathcal L\otimes K_p\ \longrightarrow\ \mathcal V\otimes K_p,\\
\boldsymbol\gamma\ :\ &
\mathcal V\ \longrightarrow\ \mathcal Q\ \longrightarrow\ K_p.
\end{aligned}
$$
These maps define a holomorphic family of Higgs bundles on
$\mathcal V\oplus\OO_{\mathscr S_{T_d}}$, with
$\boldsymbol\gamma\boldsymbol\beta\,=\,0$. By
Proposition~\ref{prop:extension}, each fiber is stable and has
Toledo invariant zero. We therefore obtain the holomorphic
classifying map
$$
  \Psi_d\ :\ T_d\ \longrightarrow\ \mathscr H_0^{\mathrm{st}},
$$
where $\mathscr H_0^{\mathrm{st}}$ is the universal moduli space over
$\Teich$ of stable projective $\PU$--Higgs bundles with Toledo
invariant zero. The construction of this map for fixed $J$ is also
given in \cite[Lemma~5.1]{McI23}.
\medskip 

The addition map induces a holomorphic map
$$
  \widetilde a_{d,b}\ :\ \mathscr F_{d,b}\ \longrightarrow\ T_d,
  \ \ \ (J,\,C,\,A_1,\,A_2,\,\xi)\ \longmapsto\ (J,\,C+A_1,\,C+A_2,\,\xi).
$$
Denote
$$
  \mathscr F_{d,b}^{\mathrm{good}}\ :=\
  (\Psi_d\,\circ\,\widetilde a_{d,b})^{-1}(\Hgood).
$$
Since the good locus is open, this is an open subset of
$\mathscr F_{d,b}$. Define
\begin{equation}
\label{eq:holonomy-map}
  H_{d,b}\ =\ \operatorname{pr}_{\X}
  \circ\mathcal N
  \circ\Psi_d
  \circ\widetilde a_{d,b}
  \big|_{\mathscr F_{d,b}^{\mathrm{good}}}\ :\ \mathscr F_{d,b}^{\mathrm{good}}
  \ \longrightarrow\ \X ,
\end{equation}
where $\mathcal N$ is the map in \eqref{eq:universal-hodge} and
$\operatorname{pr}_{\X}$ is the projection to the second factor in \eqref{eq:universal-hodge}.

\begin{proposition}
\label{prop:holonomy}
The maps $H_{d,b}$ are smooth. For $1\,\leq\,b\,\leq\,2g-3$,
their images satisfy
\begin{equation}
\label{eq:incidence-cover}
  \mathcal B_b\ =\ \bigcup_{d=b}^{2g-3}
  H_{d,b}(\mathscr F_{d,b}^{\mathrm{good}}),
\end{equation}
where $\mathcal B_b$ is defined in \eqref{eq:exceptional}.
\end{proposition}

\begin{proof}
The maps $\Psi_d$ and $\widetilde a_{d,b}$ are holomorphic, while
$\mathcal N$ is smooth by Proposition~\ref{prop:analytic-inputs}.
Consequently, the map $H_{d,b}$ constructed in
\eqref{eq:holonomy-map} is smooth.

Take $[\rho]\,\in\,\mathcal B_b$. By the definition of
$\mathcal B_b$ in \eqref{eq:exceptional}, there is a marked
equivariant minimal map $(J,\,\rho,\,f)$ with
$\ram(f)\,\geq\,b$. Let $D_1$ and $D_2$ be the effective
divisors given by Proposition~\ref{prop:extension}, and let $d$
be their common degree. In view of \eqref{eq:d-range} and
\eqref{eq:branch-divisor}, we have
$$
b\ \leq\ \deg\gcd(D_1,\,D_2)\ \leq\ d\ \leq\ 2g-3.
$$
Choose an effective divisor $C$ of degree $b$ such that
$C\,\leq\,\gcd(D_1,\,D_2)$, and set $A_i\,:=\,D_i-C$
for $i\,=\,1,\,2$. If $\xi$ is the extension class in
\eqref{eq:extension-class}, then
$$
(J,\,C,\,A_1,\,A_2,\,\xi)\ \in\
\mathscr F_{d,b}^{\mathrm{good}}.
$$
The image of this point under the map $H_{d,b}$ is $[\rho]$.

Conversely, let $(J,\,C,\,A_1,\,A_2,\,\xi)$ be a point of
$\mathscr F_{d,b}^{\mathrm{good}}$. The corresponding divisors
$D_i\,:=\,C+A_i$ satisfy the condition $C\,\leq\,\gcd(D_1,\,D_2)$.
Therefore, by \eqref{eq:branch-divisor}, the associated minimal
map $f$ satisfies the condition
$$
\ram(f)\ =\ \deg\gcd(D_1,\,D_2)\ \geq\ \deg C\ =\ b.
$$
It follows that $H_{d,b}(J,\,C,\,A_1,\,A_2,\,\xi)$ belongs
to $\mathcal B_b$. This proves \eqref{eq:incidence-cover}.
\end{proof}

\begin{proof}[Proof of Theorem~\ref{thm:ramification}]
Let $b$ be an integer with $1\,\leq\,b\,\leq\,2g-3$.
For each $d$ with $b\,\leq\,d\,\leq\,2g-3$, the space
$\mathscr F_{d,b}^{\mathrm{good}}$ is an open subset of the
complex manifold $\mathscr F_{d,b}$. Hence
Proposition~\ref{prop:incidence} implies that
$$
\dim_{\RR}\mathscr F_{d,b}^{\mathrm{good}}
\ =\ 2(8g-8-b)\ =\ N-2b
$$
whenever this open subset is nonempty, where $N$ is defined in
\eqref{eq:ambient-dim}.

The map $H_{d,b}$ is smooth by Proposition~\ref{prop:holonomy}.
On a sufficiently small relatively compact coordinate
neighborhood, a smooth map is Lipschitz. Since the source is
second countable, it can be covered by countably many such
neighborhoods. A Lipschitz map does not increase Hausdorff
dimension, and the Hausdorff dimension of a countable union is
the supremum of the dimensions of its members. Consequently,
$$
\dim_{\mathrm H}
H_{d,b}(\mathscr F_{d,b}^{\mathrm{good}})
\ \leq\ N-2b.
$$
Using \eqref{eq:incidence-cover}, it follows that
$$
\dim_{\mathrm H}\mathcal B_b\ \leq\ N-2b,
$$
which is \eqref{eq:main-bound}.

To prove the remaining assertions, choose a countable compact
exhaustion $(K_{d,j})_{j\geq1}$ of each
$\mathscr F_{d,b}^{\mathrm{good}}$. Such an exhaustion exists
because this space is a second-countable finite-dimensional
manifold. Equation~\eqref{eq:incidence-cover} now gives
$$
\mathcal B_b\ =\
\bigcup_{d=b}^{2g-3}\ \bigcup_{j\geq1}H_{d,b}(K_{d,j}).
$$
Each set $H_{d,b}(K_{d,j})$ is compact, and its Hausdorff
dimension is at most $N-2b$, which is strictly less than $N$.
It therefore has zero volume for every smooth Riemannian metric
on $\X$. It also has empty interior, since every nonempty open
subset of an $N$-dimensional smooth manifold has Hausdorff
dimension $N$. Being compact, it is closed and hence it is nowhere
dense. Thus $\mathcal B_b$ is a countable union of compact null
sets, and it is $F_\sigma$, null, and meagre.

Finally, for any marked equivariant minimal map with holonomy
in $\X$, Proposition~\ref{prop:extension} gives the following:
$$
\ram(f)\ =\ \deg\gcd(D_1,\,D_2)\ \leq\ d\ \leq\ 2g-3.
$$
Therefore $\mathcal B_b\,=\,\varnothing$ for $b\,>\,2g-3$.
\end{proof}

\begin{proof}[Proof of Corollary~\ref{cor:generic-immersion}]
By Theorem~\ref{thm:ramification}, the set $\mathcal B_1$ is
$F_\sigma$ and null in $\X$. Its complement is therefore
$G_\delta$ and has full measure. Every nonempty open subset
of $\X$ has positive smooth volume, so this complement is
dense. For $[\rho]\,\in\,\X\setminus\mathcal B_1$, the
definition of $\mathcal B_1$ excludes every marked equivariant
minimal map with positive total ramification. Consequently,
every such map is an immersion, for every
$J\,\in\,\Teich$.
\end{proof}

\section{Energy, compactness, and generic finiteness}
\label{sec:energy}

Recall the energy function $E_\rho$ in \eqref{eq:energy} and
the convex cocompact locus $\CC$ in \eqref{eq:cc-locus}.
For $[\rho]\,\in\,\X\setminus\mathcal B_1$, every critical
point of $E_\rho$ determines a minimal immersion. 

\begin{proposition}
\label{prop:energy}
Let $[\rho]\,\in\,\X\setminus\mathcal B_1$ and
$J\,\in\,\Crit(E_\rho)$. Then
\begin{equation}
\label{eq:area-bound}
E_\rho(J)\ =\ \Area(f_{J,\rho})\ \leq\ 4\pi(g-1).
\end{equation}
If $[\rho]\,\in\,\CC\setminus\mathcal B_1$, then
$\Crit(E_\rho)$ is compact and nonempty. Moreover, the map
\begin{equation}
\label{eq:critical-fiber-homeo}
\Crit(E_\rho)\ \longrightarrow\
\pi_{\mathrm{imm}}^{-1}([\rho]),\qquad
J\ \longmapsto\ (J,\,[\rho])
\end{equation}
is a homeomorphism.
\end{proposition}

\begin{proof}
Let $f\,:=\,f_{J,\rho}$. Since
$J\,\in\,\Crit(E_\rho)$, Remark~\ref{eq:critical-conformal}
shows that $f$ is weakly conformal. It is nonconstant by
irreducibility, and it has no branch points because
$[\rho]\,\notin\,\mathcal B_1$. Therefore the induced metric
$h\,:=\,f^*g_{\CH}$ descends to a Riemannian metric on
$S$. 

Let $\mathrm{II}$ denote the second fundamental form
of $f$, and let $\{e_1,\ e_2\}$ be a local
$h$-orthonormal frame on $S$. The Gauss equation gives
\[
K_h\ =\ K_{\CH}(df(TS))
  +\langle \mathrm{II}(e_1,\ e_1),
           \mathrm{II}(e_2,\ e_2)\rangle -\|\mathrm{II}(e_1,\ e_2)\|^2.\]
Since $f$ is minimal, its mean curvature vector
vanishes, and hence
$\mathrm{II}(e_1,\ e_1)+\mathrm{II}(e_2,\ e_2)=0.
$ Consequently, using the symmetry of $\mathrm{II}$,
we obtain the following:
\[
K_h\ =\ K_{\CH}(df(TS))
-\frac12\|\mathrm{II}\|_h^2.
\]
Since the sectional curvatures of $\CH$ lie in the interval
$[-4,-1]$, and $\|\mathrm{II}\|_h^2\,\geq\, 0$, it follows that $K_h\,\leq\, -1.$  For a conformal map, the energy in \eqref{eq:energy} coincides with
the area \cite{DW07}. This proves \eqref{eq:area-bound}.

\medskip

Now assume that $[\rho]\,\in\,\CC\setminus\mathcal B_1$.
By \cite[Theorem~A]{GW07}, the function
$E_\rho\,:\,\Teich\,\longrightarrow\,\RR$ is proper.
Consequently, the sublevel set
$$
E_\rho^{-1}\bigl([0,\,4\pi(g-1)]\bigr)
$$
is compact. By \eqref{eq:area-bound}, it contains every
critical point of $E_\rho$. The critical set is closed,
since $E_\rho$ is smooth. It now follows that
$\Crit(E_\rho)$ is compact.

To see that this set is nonempty, choose
$J_0\,\in\,\Teich$. The sublevel set
$$
E_\rho^{-1}\bigl([0,\,E_\rho(J_0)]\bigr)
$$
is nonempty and compact. The restriction of $E_\rho$ to this
set attains its minimum, and every point outside the set has
larger energy. This minimum is therefore a global minimum on
$\Teich$, so it is a critical point of $E_\rho$.

Finally, by Remark~\ref{eq:critical-conformal}, and the absence
of branch points, the points of $\Crit(E_\rho)$ are precisely
the complex structures $J$ for which
$(J,\,[\rho])\,\in\,\Mimm$. Thus the map in
\eqref{eq:critical-fiber-homeo} is a bijection. This map is
the restriction of the inclusion map
$J\,\longmapsto\,(J,\,[\rho])$, and its inverse is the
restriction of the projection to $\Teich$. Both are continuous
for the subspace topologies, so the bijection is a
homeomorphism.
\end{proof}

\begin{proof}[Proof of Theorem~\ref{thm:finiteness}]
Consider the holonomy projection
$$
\pi_{\mathrm{imm}}\ :\ \Mimm\ \longrightarrow\ \X
$$
in \eqref{eq:immersion-projection}. By
Proposition~\ref{prop:analytic-inputs}, it is smooth and its domain has complex dimension $8(g-1)$. Hence
$$
\dim_{\RR}\Mimm\ =\ 16g-16\ =\ \dim_{\RR}\X\ =\ N.
$$

Let $Z\,\subset\,\Mimm$ be the set of critical points of
$\pi_{\mathrm{imm}}$, and let
$\mathcal S\,:=\,\pi_{\mathrm{imm}}(Z)$ be its set of
critical values. Sard's theorem implies that $\mathcal S$
is null in $\X$. The set $Z$ is closed, because the
differential of $\pi_{\mathrm{imm}}$ is continuous. Choose
a countable compact exhaustion $(L_j)_{j\geq1}$ of the
second-countable manifold $\Mimm$. Then
$$
\mathcal S\ =\
\bigcup_{j\geq1}\pi_{\mathrm{imm}}(Z\cap L_j).
$$
Each subset in this union is compact and null. It has empty
interior because every nonempty open subset of $\X$ has
positive smooth volume. It is therefore nowhere dense.
Consequently, $\mathcal S$ is $F_\sigma$, null, and meagre
in $\X$.

Define
\begin{equation}
\label{eq:generic-set}
\mathcal G\ :=\
\CC\setminus(\mathcal B_1\cup\mathcal S),
\end{equation}
where $\CC$ is defined in \eqref{eq:cc-locus}.
Theorem~\ref{thm:ramification} shows that $\mathcal B_1$ is
also $F_\sigma$, null, and meagre. Since $\CC$ is open in
$\X$, the subset $\mathcal G$ is $G_\delta$ relative to $\CC$
and has full measure there. Every nonempty relatively open
subset of $\CC$ is open in $\X$ and has positive smooth
volume. Hence $\mathcal G$ is actually dense in $\CC$.

Let $[\rho]\,\in\,\mathcal G$. Since
$[\rho]\,\notin\,\mathcal B_1$, every nonconstant
equivariant minimal map with this holonomy is an immersion.
Since $[\rho]\,\notin\,\mathcal S$, it is a regular value
of $\pi_{\mathrm{imm}}$. The source and target have the
same real dimension $N$, so the regular-value theorem gives
that $\pi_{\mathrm{imm}}^{-1}([\rho])$ is a
zero-dimensional smooth manifold. In particular, this fiber
is discrete.

By Proposition~\ref{prop:energy}, this fiber is homeomorphic to the
compact nonempty set $\Crit(E_\rho)$. It is therefore
compact, discrete, and nonempty. A compact discrete space
is finite: the cover by its open singletons has a finite
subcover. Thus
$$
1\ \leq\ \#\Crit(E_\rho)
\ =\ \#\pi_{\mathrm{imm}}^{-1}([\rho])\ <\ \infty,
$$
which proves \eqref{eq:finite-fiber}.
\end{proof}

\end{document}